\pdfoutput=1
\documentclass[11pt, oneside]{article}   	
\usepackage[margin=0.5in]{geometry}
\usepackage{amsmath}
\usepackage{amsthm}
\usepackage{amssymb}

\usepackage{mathtools}
\usepackage{yfonts}
\usepackage{bbm}
\usepackage{enumitem}
\usepackage{tikz-cd}
\usetikzlibrary{cd}
\setlist{nolistsep,leftmargin=*}
\usepackage{listings}
\usepackage{bm}
\usepackage{hyperref}

\theoremstyle{definition} 
\newtheorem{thm}{Theorem}[section]
\newtheorem{lem}[thm]{Lemma}
\newtheorem{prop}[thm]{Proposition}
\newtheorem{cor}[thm]{Corollary}
\newtheorem{exmp}[thm]{Example}

\newtheorem{defn}[thm]{Definition}
\newtheorem*{thm*}{Theorem}
\newtheorem*{prop*}{Proposition}
\theoremstyle{remark}
\newtheorem*{rmk}{Remark}

\newcommand{\bmat}{\begin{matrix}}
\newcommand{\emat}{\end{matrix}}
\newcommand{\bpmat}{\begin{pmatrix}}
\newcommand{\epmat}{\end{pmatrix}}
\renewcommand{\ni}{\noindent}
\renewcommand{\t}{\text}
\newcommand{\ra}{\longrightarrow}
\newcommand{\xra}{\xrightarrow}
\newcommand{\mbb}{\mathbb}
\newcommand{\mc}{\mathcal}
\newcommand{\mf}{\mathfrak}

\newcommand{\tb}{\textbf}
\newcommand{\ti}{\textit}
\newcommand{\ol}{\overline}
\renewcommand{\epsilon}{\varepsilon}

\newcommand{\rlim}{\varinjlim}
\newcommand{\llim}{\varprojlim}
\newcommand{\spec}{\t{Spec }}

\renewcommand{\mod}{\t{mod-}}
\newcommand{\Mod}{\t{Mod-}}
\newcommand{\Hom}{\t{Hom}}
\renewcommand{\phi}{\varphi}

\newcommand{\hcm}[1]{\t{CohCM}(#1)}
\newcommand{\lcm}[1]{\rlim{\t{CM}(#1)}}
\newcommand{\cm}[1]{\t{CM}(#1)}
\newcommand{\hra}{\hookrightarrow}
\newcommand{\tra}{\twoheadrightarrow}

\begin{document}

\title{Two definable subcategories of maximal Cohen-Macaulay modules}

\author{Isaac Bird
\\
\small{School of Mathematics, University of Manchester, Oxford Road, Manchester M13 9PL}
\\
\small{\url{isaac.bird@manchester.ac.uk}}}
\date{}

\maketitle
\begin{abstract}
Over a Cohen-Macaulay ring we consider two extensions of the maximal Cohen-Macaulay modules from the viewpoint of definable subcategories, which are closed under direct limits, direct products and pure submodules. After presenting these categories, we compare them and consider which properties they inherit from the maximal Cohen-Macaulay modules. We then consider some further properties of these classes and how they interact with the entire module category.
\end{abstract}

\section{Introduction}
\ni
Over a commutative Noetherian local ring $(R,\mf{m},k)$ of Krull dimension $d$, a finitely generated $R$-module $M$ is \ti{maximal Cohen-Macaulay} if every maximal $M$-sequence contained in $\mf{m}$ has length $d$. A classic theorem due to Rees shows that this is equivalent to asking for $\t{Ext}_{R}^{i}(k,M)=0$ for all $i<d$. These equivalences fail when the assumption of $M$ being finitely generated is removed, so there is no uniform way to define the maximal Cohen-Macaulay property beyond finitely generated $R$-modules. We investigate two possible extensions from the perspective of \ti{definable subcategories}.
\\
\\
A definable subcategory is a class of modules that is closed under pure submodules, direct limits and direct products. These classes are in bijection with the closed sets of a topological space, called the Ziegler spectrum, whose underlying set is the set of isomorphism classes of indecomposable pure-injective modules. Given any class of modules, there is a smallest definable subcategory containing it, which then corresponds to a closed subset of the Ziegler spectrum. Understanding the closed set corresponding to the class of maximal Cohen-Macaulay modules provides one motivation for this approach. 
\\
\\
By considering the Ext-depth of a module, as described in \ref{strook}, we consider the category
\begin{equation}\label{cohcm}
\hcm{R}:=\{M\in R\t{-Mod} :\t{Ext}_{R}^{i}(k,M)=0 \t{ for all $i<d$}\}
\end{equation}
consisting of all $R$-modules whose Ext-depth is at least $d$. This is a definable subcategory and it is clear that its finitely generated modules are precisely the maximal Cohen-Macaulay modules. However, it is not clear that this is the smallest definable subcategory with this property. 
\\
\\
Since definable subcategories are closed under direct limits, the class $\lcm{R}$, which consists of all modules that can be obtained as direct limits of systems of maximal Cohen-Macaulay modules, will be contained inside the class defined in (\ref{cohcm}). However, $\lcm{R}$ is not necessarily always a definable class, since it need not be closed under products. In \cite[Theorem B]{holm}, H. Holm showed that whenever $R$ admits a canonical module $\Omega$, the class $\lcm{R}$ is a definable subcategory, and is the smallest definable subcategory extending the maximal Cohen-Macaulay modules, as well as providing several equivalent characterisations of the modules in this class. Using Holm's characterisations, we are able to determine when the categories $\hcm{R}$ and $\lcm{R}$ coincide.

\begin{thm*}[\ref{sep}]
Let $R$ be a Cohen-Macaulay ring. If $\t{dim}\,R$=1, then $\hcm{R}=\rlim \cm{R}$. Moreover, if $R$ admits a canonical module $\hcm{R}=\lcm{R}$ if and only if $\t{dim}\,R=1$.
\end{thm*}
\ni
In assuming the existence of a canonical module, we can consider the canonical duality $\Hom_{R}(-,\Omega)$. This functor plays a fundamental role in understanding the category of maximal Cohen-Macaulay $R$-modules, $\cm{R}$, as illustrated in the monographs \cite{lw} and \cite{yosh}. Since both $\hcm{R}$ and $\lcm{R}$ are extensions of $\cm{R}$, we aim to see which, if any, properties of this functor extend to these categories from $\cm{R}$. In particular, we show the following.

\begin{thm*}[\ref{hom},\ref{hom2}]
If $R$ is a Cohen-Macaulay ring admitting a canonical module $\Omega$, then $\Hom_{R}(-,\Omega)$ is an endofunctor on both $\rlim \cm{R}$ and $\hcm{R}$. Moreover, $\Omega$ is an injective object in both these categories.
\end{thm*}
\ni
One can partition both $\lcm{R}$ and $\hcm{R}$ into the modules of finite and infinite Ext-depth. In the case of $\lcm{R}$, it is clear from Holm's characterisation that the modules in $\lcm{R}$ whose Ext-depth is finite are precisely Hochster's balanced big Cohen-Macaulay modules, see \cite[Ch. 8]{bh}. We then consider how the canonical dual acts on this category. The modules of infinite depth are of interest in their own right, as they form a definable subcategory that contains no finitely generated modules but is still very much related to $\cm{R}$. We provide an in-depth example by considering the one-dimensional $A_{\infty}$ singularity.
\\
\\
We then consider some of the properties of both $\hcm{R}$ and $\lcm{R}$, both categorical and homological. Since both these categories are definable, they are already covering and preenveloping. However, we are able to improve on this by replicating Holm's result \cite[Theorem D]{holm} to show the following:

\begin{thm*}[\ref{cotorsion}]
Let $R$ be a Cohen-Macaulay ring. Then $(\hcm{R},\hcm{R}^{\perp})$ is a perfect hereditary cotorsion pair in $\Mod{R}$.
\end{thm*}
\ni
This enables us to consider the $\hcm{R}$-dimension of a module, which is closely related to its Ext-depth. We also turn our attention to inverse limits and see how the inverse limit closure of $\cm{R}$ is related to $\lcm{R}$.
\\
\\
Lastly, we look at the special case when $\t{dim }R=1$, where $\lcm{R}$ and $\hcm{R}$ coincide. In this situation the modules of infinite Ext-depth have a particularly rich structure.
\begin{thm*}[\ref{abelian}]
If $R$ is a one-dimensional Cohen-Macaulay ring, then the class of infinite depth modules is a Grothendieck Abelian category.
\end{thm*}

\section*{Acknowledgements}
\ni
I am grateful to Mike Prest for his comments, feedback and suggestions in relation to this work, as well as the referee for their constructive comments and the alternative proof of \ref{hom2}. This research was funded by the University of Manchester.

\section{Background: depth and duality over local rings}
\ni
For this section, $(R,\mf{m},k)$ is a commutative Noetherian local ring of Krull dimension $d$. For an $R$-module $M$, an \ti{$M$-regular sequence}, or simply $M$-\ti{sequence}, is a sequence of elements $\bm{x}=x_{1},\cdots, x_{n}\in \mf{m}$ such that the following conditions hold:
\begin{enumerate}
\item $x_{i}$ is a non-zerodivisor on $M/(x_{1},\cdots,x_{i-1})$ for all $i=1,\cdots, n$;
\item $M/\bm{x}M\neq 0$.
\end{enumerate}
\ni
If only the first condition holds, we say that $\bm{x}$ is a \ti{weak $M$-sequence}. Nakayama's lemma shows that whenever $M$ is a non-zero finitely generated $R$-module all weak $M$-sequences are automatically $M$-sequences. This is not the case for arbitrary $R$-modules. We say that a regular sequence is \ti{maximal} if it cannot be extended. 
\\
\\
When $M\in\mod{R}$, the category of finitely generated $R$-modules, a theorem due to Rees shows that all maximal $M$-sequences in $\mf{m}$ are of the same length.

\begin{thm}{\cite[Theorem 1.2.5]{bh}}
Let $(R,\mf{m},k)$ be a Noetherian local ring and $M$ a finitely generated $R$-module. Then all maximal $M$-sequences in $\mf{m}$ have the same length $n$, given by
$$n = \t{inf}\{i:\t{Ext}_{R}^{i}(k,M)\neq 0\}.$$
\end{thm}
\begin{defn}
With the notation of the above theorem, we call the common length of all maximal $M$-sequences to be the \ti{depth} of $M$, and we denote it $\t{dp }M$.
\end{defn}
\ni
\ni
We state the following for clarity: if $M$ is a finitely generated $R$-module, then 
\begin{equation*}
\t{dp }M = \t{inf}\{i:\t{Ext}_{R}^{i}(k,M)\neq 0 \}.
\end{equation*}
\begin{defn}
Let $R$ be as above. We say that a non-zero finitely generated $R$-module is \ti{maximal Cohen-Macaulay}, or simply \ti{Cohen-Macaulay}, if the equivalent following conditions hold:
\begin{enumerate}
\item $\t{dp }M=d$;
\item $\t{Ext}_{R}^{i}(k, M)=0$ for all $0\leq i<d$;
\item $H_{\mf{m}}^{i}(M)=0$ for all $i\neq d$.
\end{enumerate}
We say that $R$ is a \ti{Cohen-Macaulay ring} if it is a Cohen-Macaulay module over itself. For convention we assume the zero module is Cohen-Macaulay.
\end{defn}
\ni
Recall that for any $R$-module $N$, $H_{\mf{m}}^{i}(N)$ is the $i$-th \ti{local cohomology} of $N$ (with support in $\mf{m}$), and is given by the formula 
$$H_{\mf{m}}^{i}(N)=\rlim_{t} \t{Ext}_{R}^{i}(R/\mf{m}^{t},N)$$
for every $i\geq 0$. 
\\
\\
If one wishes to extend the definition of Cohen-Macaulay from $\mod{R}$ to $\Mod{R}$, the class of all $R$-modules, some immediate obstructions arise. For instance, if $M$ is any $R$-module, the maximal $M$-sequences in $\mf{m}$ may no longer have the same length. Moreover, the equivalences of the above three conditions fails, as illustrated in the following example, due to Strooker.

\begin{exmp}{\cite[p. 91]{strooker}}
Let $k$ be a field and $R=k[[x,y]]$. Set $M=\bigoplus R/(f)$, where the sum runs over all elements of $\mf{m}$. Then the depth of $M$ is zero, since no element of $\mf{m}$ is regular on $M$. Since $R$ is a domain, the principal ideal $(f)$ is free for every $f\in\mf{m}$, so is Cohen-Macaulay. Applying the functor $\t{Hom}_{R}(k,-)$ to the short exact sequence $0\ra (f)\ra R\ra R/(f)\ra 0$ shows that $\Hom_{R}(k, R/(f))=0$. Consequently $\Hom_{R}(k,M)=0$ and $\inf\{i\geq 0:\t{Ext}_{R}^{i}(k,M)\neq 0\}\neq 0$. In fact, the latter is actually equal to one. 
\end{exmp} 
\ni
In light of this, the notion of depth has been generalised using the invariant given in Rees's theorem.
\begin{defn}{\cite[5.3.6]{strooker}} 
Let $M$ be an arbitrary $R$-module. We define the \ti{Ext-depth} of $M$, denoted $\t{E-dp }M$ as
$$\t{E-dp }M=\inf\{i\geq 0:\t{Ext}_{R}^{i}(k,M)\neq 0\}.$$
If the above integer does not exist, we say that the module has \ti{infinite Ext-depth}.
\end{defn}
\ni
Fortunately, the relationship between local cohomology and the Ext-functors does not restrict to finitely generated modules, so we can also use local cohomology to measure Ext-depth. More precisely, for all $R$-modules $M$, one has
$$\t{E-dp }M=\inf\{i\geq 0: H_{\mf{m}}^{i}(M)\neq 0\}.$$
One can find a proof of this at {\cite[Prop. 5.3.15]{strooker}} or {\cite[Thm. 9.1]{24h}}. In fact, this relationship extends to complexes of modules, as illustrated in \cite{fi}, but this setting will not be used. This will at times have its advantages, due to properties of local cohomology, in particular that $H_{\mf{m}}^{i}(N)=0$ for all $i>\t{dim }R$ and all $R$-modules $N$ (see {\cite[6.1.2 Grothendieck's Vanishing Theorem]{lc}}). Consequently, if a module has finite Ext-depth, it is at most the Krull dimension of $R$. Much more information about local cohomology can be found in \cite{lc} and \cite{24h}.
\\
\\
One can relate Ext-depth and depth for arbitrary $R$-modules. Indeed, for any $R$-module $M$, there is an inequality $\t{dp }M\leq \t{E-dp }M$. Moreover, if $\t{E-dp}(M)$ is finite, it is equal to $\t{dp }M$ if and only if there is an $M$-sequence $\bm{x}=x_{1},\cdots, x_{s} \in \mf{m}$ and a non-zero element $z\in M/\bm{x}M$ such that $\mf{m}z\neq 0$. Proofs of these claims can be found at \cite[5.3.7, 5.3.8]{strooker}.
\\
\\
Returning to the finitely generated case, in the situation where $R$ is a Cohen-Macaulay ring we let $\cm{R}$ denote the full subcategory of $\mod{R}$ consisting of the Cohen-Macaulay modules. This category has been extensively studied and is well understood, as can be seen in the texts \cite{yosh} and \cite{lw}. 
\\
\\
A class of Cohen-Macaulay rings that will be of particular interest to us will be those that admit a canonical module.
\begin{defn}
If $(R,\mf{m},k)$ is a Cohen-Macaulay local ring, then a Cohen-Macaulay $R$-module $\Omega$ is said to be a \ti{canonical module} if 
$$
\t{dim}_{k} \ \t{Ext}_{R}^{i}(k,\Omega) = 
\begin{cases}
0 & \mbox{ if }i \neq \t{dim }R, \\
1 & \mbox{ if }i = \t{dim }R.
\end{cases}
$$
\end{defn}
\ni
It is known that, if it exists, the canonical module is unique up to isomorphism, see \cite[Theorem 3.3.4]{bh}. Necessary and sufficient conditions for the existence of the canonical module can be found at {\cite[Thm 3.3.6]{bh}}; in particular, any complete local ring admits a canonical module. A notable subclass of Cohen-Macaulay rings admitting a canonical module are Gorenstein rings, for which we recall the definition.

\begin{defn}
A Cohen-Macaulay ring $R$ is \ti{Gorenstein} if it has finite injective dimension over itself.
\end{defn}
\ni
There is alternative property that completely determines Gorenstein rings, that will be of some use.
 
\begin{prop}{\cite[3.3.7]{bh}}
The following are equivalent for a Cohen-Macaulay ring $R$.
\begin{enumerate}
\item $R$ is Gorenstein.
\item $\Omega$ exists and is isomorphic to $R$.
\end{enumerate}
\end{prop}
\ni
Over a Cohen-Macaulay ring admitting a canonical module $\Omega$, the functor $(-)^{*}:=\Hom_{R}(-,\Omega)$ plays a special role in understanding the category $\cm{R}$, in particular its Auslander-Reiten theory. This is illustrated in great detail in the monograph \cite{yosh}. While we will not need this much detail, there are a few properties of the functor that we will use. 

\begin{prop}{\cite[3.3.10]{bh}}\label{ld}
Let $(R,\mf{m},k)$ be a Cohen-Macaulay ring admitting a canonical module $\Omega$ and let $M\in\cm{R}$. Then
\begin{enumerate}
\item $M^{*}=\Hom_{R}(M,\Omega)$ is a Cohen-Macaulay module,
\item $\t{Ext}_{R}^{i}(M,\Omega)=0$ for all $i>0$,
\item The natural map $M\ra M^{**}$ is an isomorphism.
\end{enumerate}
In particular, $\Omega$ is an injective cogenerator in $\cm{R}$.
\end{prop}

\ni
Over a complete local ring, one can relate the canonical module with local cohomology using the following theorem.

\begin{thm}[Grothendieck local duality]{\cite[(proof of) 3.5.8]{bh}}\label{gd}
Let $(R,\mf{m},k)$ be a complete Cohen-Macaulay ring of Krull dimension $d$. Then for all $R$-modules $M$ and integers $i$ there is a natural isomorphism
$$\t{Ext}_{R}^{i}(M,\Omega)\simeq \Hom_{R}(H_{\mf{m}}^{d-i}(M),E(k)),$$
where $E(k)$ denotes the injective envelope of the residue field $k$.
When $M$ is finitely generated, there is a further isomorphism
$$H_{\mf{m}}^{d-i}(M)\simeq \Hom_{R}(\t{Ext}_{R}^{i}(M,\Omega),E(k)).$$
\end{thm}
\ni
If $M$ is any $R$-module, we will call the module $M^{\vee}:=\Hom_{R}(M,E(k))$ the \ti{Matlis dual} of $M$. We note that $E(k)$ is an injective cogenerator in $\Mod{R}$. 
\\
\\
The following result gives a few useful properties of the Matlis dual.

\begin{prop}{\cite[3.4.1, 3.4.5-7]{rha}}
Let $(R,\mf{m},k)$ be a complete Noetherian local ring and $E(k)$ as above.
\begin{enumerate}
\item For every $R$ module the canonical map $M\ra M^{\vee\vee}$ is injective.
\item If $M$ is either a finitely generated or artinian, then $M$ is reflexive, that is the canonical map $M\ra M^{\vee\vee}$ is an isomorphism;
\item $(-)^{\vee}$ gives a duality between finitely generated $R$-modules and artinian $R$-modules.
\end{enumerate}
\end{prop}
\ni
The Matlis dual also gives us the following useful relations between Ext and Tor modules.

\begin{lem}{\cite[1.2.11]{trlifaj}}\label{relations}
Let $(R,\mf{m},k)$ be a commutative Noetherian local ring and $E(k)$ as above.
\begin{enumerate}
\item Let $M$ and $N$ be arbitrary $R$-modules, then for any $i\geq 0$ there is an isomorphism
$$\t{Ext}_{R}^{i}(M,N^{\vee})\simeq (\t{Tor}_{i}^{R}(M,N))^{\vee}.$$

\item Let $M$ be a finitely generated module and $N$ an arbitrary module. Then for any $i\geq 0$ there is an isomorphism
$$\t{Tor}_{i}^{R}(M, N^{\vee})\simeq (\t{Ext}_{R}^{i}(M,N))^{\vee}.$$
\end{enumerate}
\end{lem}
\ni
The above result is a specialisation of a much more general result, which is stated in its completeness in the given reference. Using the above dualities, we consider a dual notion to $\t{E-dp}$, again following the terminology of Strooker.

\begin{defn}{\cite[p. 102]{strooker}}
Let $M$ be an $R$-module. Define the \ti{Tor-codepth} of $M$, denoted by $\t{T-codp }M$ as
$$\t{T-codp }M=\inf\{i\geq 0:\t{Tor}_{i}^{R}(k,M)\neq 0\}.$$
If no such integer exists, we say the module has \ti{infinte Tor-codepth}.
\end{defn}
\begin{rmk}
Tor-codepth is also known as \ti{width}, see \cite[1.9]{duality} and \cite{yass}.
\end{rmk}
\ni
From the above lemma, it is clear that if $\t{E-dp }M=t$, then $\t{T-codp }M^{\vee}=t$, and vice-versa, where $t$ can be either finite or infinite. One can generalise the notions of Ext-depth and Tor-codepth as follows: if $\mf{a}\subset R$ is an ideal, define
$$\t{E-dp}(\mf{a},M)=\inf\{n\geq :\t{Ext}_{R}^{n}(R/\mf{a},M)\neq 0\}$$
and the dual notion for $\t{T-codp}(\mf{a},M)$. We can relate Ext-depth and Tor-depth using the following useful result.

\begin{prop}{\cite[Cor. 6.1.8]{strooker}}\label{strook}
Let $R$ be a complete Noetherian local ring, $\mf{a}$ an ideal and $M$ an $R$-module. Then $\t{E-dp}(\mf{a},M)<\infty$ if and only if $\t{T-codp}(\mf{a},M)<\infty$, and if this is the case then
\begin{equation*}
\t{E-dp}(\mf{a},M)+\t{T-codp}(\mf{a},M)\leq \t{dim }R.
\end{equation*}
\end{prop}

\ni
We can use the left exactness of the Hom functor to see how Ext-depth behaves with respect to short exact sequences.

\begin{lem}[Depth lemma]{\cite[Prop. 9.1.2(e)]{bh}}\label{dl}
Let $\mf{a}$ be an ideal of $R$ and $0\ra L\ra M \ra N\ra 0$ a short exact sequence of $R$-modules. Then
$$\t{E-dp}(\mf{a}, M)\geq \min\{\t{E-dp}(\mf{a},L),\t{E-dp}(\mf{a},N)\}$$
$$\t{E-dp}(\mf{a},L)\geq \min\{\t{E-dp}(\mf{a},M),\t{E-dp}(\mf{a},N)+1\}$$
$$\t{E-dp}(\mf{a},N)\geq \min\{\t{E-dp}(\mf{a},L)-1,\t{E-dp}(\mf{a},M)\}$$
\end{lem}

\section{Two definable subcategories of Cohen-Macaulay modules}
\ni
Throughout this section, we will assume that $(R,\mf{m},k)$ is a Cohen-Macaulay ring of Krull dimension $d$. We will not assume the existence of a canonical module. As illustrated in the previous section, there is an ambiguity when extending the definition of a Cohen-Macaulay $R$-module. Since for finitely generated modules depth and Ext-depth coincide, one can consider the subcategory of $R$-modules that satisfy the Ext-depth definition of Cohen-Macaulay. We will define $\hcm{R}$ to be precisely these modules, that is
\begin{align*}\hcm{R}&=\{M\in\Mod{R}:\t{Ext}_{R}^{i}(k,M)=0 \t{ for all }i<d\} \\
& = \{M\in\Mod{R}: H_{\mf{m}}^{i}(M)=0 \t{ for all }i<d\}.
\end{align*}

\begin{lem}
$\hcm{R}$ is a definable subcategory of $\Mod{R}$, that is it is closed under pure submodules, products and direct limits.
\end{lem}

\begin{proof}
The residue field $k$ is finitely presented, so by \cite[Theorem 10.2.35.(c)]{psl}, the functors $\t{Ext}_{R}^{i}(k,-):\mod R\ra \t{Ab}$ are finitely presented for every $i\geq 0$. Set $X=\{\t{Ext}_{R}^{i}(k,-): 0\leq i<\t{dim }R\}$; so, by \cite[Cor. 10.2.32]{psl}, the subcategory $\mc{X}=\{M\in\Mod R: FX=0\t{ for all } F\in X\}$ is a definable subcategory. But $\mc{X}$ is just $\hcm{R}$ by definition.
\end{proof}

\ni
Clearly the finitely generated modules in $\hcm{R}$ are just $\cm{R}$, so if one wants to consider extensions of $\cm{R}$ that are definable, $\hcm{R}$ is a valid option. However, there will be a definable subcategory generated by $\cm{R}$, which we denote $\langle \cm{R}\rangle$. This is the smallest definable subcategory containing $\cm{R}$, and is its closure under direct limits, direct products and pure submodules. However, $\hcm{R}$ is a definable subcategory containing $\cm{R}$, so there is an inclusion of definable subcategories $\langle \cm{R}\rangle \subseteq \hcm{R}$.

\begin{prop}
Let $R$ be a Cohen-Macaulay ring with a canonical module. The definable subcategory of $\Mod{R}$ generated by $\cm{R}$ is $\lcm{R}$, the class of all $R$-modules that can be realised as a direct limit of modules in $\cm{R}$.
\end{prop}
\begin{proof}
Since $\cm{R}$ is closed under finite direct sums, it suffices, by \cite[Cor. 3.4.37]{psl}, to show that $\cm{R}$ is preenveloping in $\mod{R}$. But this is {\cite[Thm. C]{holm}}.
\end{proof}

\ni
H. Holm characterised the modules in $\lcm{R}$ and his descriptions will enable us to consider its differences with $\hcm{R}$. Before doing this, we need to recall some more definitions.

\begin{defn}
Let $A$ be a Cohen-Macaulay ring admitting a canonical module $\Omega$. The \ti{trivial extension} of $A$ by $\Omega$ is the ring $A\ltimes\Omega$ whose underlying abelian group is $A\oplus\Omega$ and whose multiplication is given by
$$(a_{1},\omega_{1})(a_{2},\omega_{2})=(a_{1}a_{2},a_{1}\omega_{2}+a_{2}\omega_{1}),$$
for any $r_{1},r_{2}\in A$ and $\omega_{1},\omega_{2}\in\Omega$.
\end{defn}
\ni
Restriction of scalars along the inclusion ring homomorphism $i:A\ra A\ltimes\Omega$ gives a functor $U:\Mod{(A\ltimes\Omega})\ra \Mod{A}$ called the \ti{underlying functor}: if $N$ is an $A\ltimes\Omega$-module, then $U(N)$ has the same abelian group as $N$, and the $A$-action is given by $a\cdot n=(a,0)n$ for any $a\in A$, $n\in N$. $U$ is an exact functor that commutes with both direct and inverse limits, see {\cite[1.6, 1.7]{fgr}}. Doing the corresponding construction for the projection $p:A\ltimes\Omega\ra A$ gives a functor $Z:\Mod{A}\ra \Mod{(A\ltimes\Omega)}$ that shares the same properties as $U$, and since $p\circ i$ is identity on $A$, the composition $UZ$ is the identity functor on $\Mod{A}$. Since we assumed $A$ is commutative, it is clear that so is $A\ltimes\Omega$. We can list some of its properties as a ring.
\begin{thm}
Let $A$ and $\Omega$ be as in the above definition.
\begin{enumerate}
\item $A\ltimes\Omega$ is a commutative Noetherian local ring.
\item $A\ltimes\Omega$ and $A$ have the same Krull dimension.
\item $A\ltimes\Omega$ is a Gorenstein local ring.
\end{enumerate}
\end{thm}
\ni
The first two results can be found in {\cite[Thm. 3.2]{ideal}}, while the third result can be found at {\cite[Theorem 7]{reiten}}. 
\\
\\
Two classes of Gorenstein modules will be particularly useful in understanding the category $\lcm{R}$. The definitions given here are not the traditional ones, which can be found in \cite[Chapter 10]{rha}, but are equivalent and suit our purposes. 

\begin{defn}
Let $R$ be a Gorenstein local ring.
\begin{enumerate}
\item \cite[11.5.3]{rha} A finitely generated $R$-module $M$ is \ti{Gorenstein projective} if $\t{Ext}_{R}^{i}(M,R)=0$ for all $i\geq 1$;
\item \cite[10.3.8]{rha} An $R$ module $M$ is \ti{Gorenstein flat} if $\t{Tor}_{i}^{R}(M,E)=0$ for all injective $R$-modules $E$ and all $i\geq 1$.
\end{enumerate}
\end{defn}
\ni
We will state a brief lemma giving some properties showing how Gorenstein projective and flat modules relate to both each other and what we have already seen.

\begin{lem}\label{props}
Let $R$ be a Gorenstein local ring.
\begin{enumerate}
\item $M$ is a Gorenstein flat $R$-module if and only if $M$ is the direct limit of a directed system of finitely presented Gorenstein projective $R$-modules.
\item Any finitely presented $R$-module is Gorenstein projective if and only if it is Cohen-Macaulay.
\end{enumerate}
\end{lem}
\ni
These statements (in more generality), and their proofs, can be found at \cite[Theorem 10.3.8.4]{rha} and \cite[Corollary 10.2.7]{rha} respectively. We are now in a position to give Holm's description of $\lcm{R}$.

\begin{thm}{\cite[Thm. B]{holm}}\label{holm}
Let $R$ be a Cohen-Macaulay ring admitting a canonical module $\Omega$. The following are equivalent for an $R$-module $M$.
\begin{enumerate}
\item $M$ is in $\lcm{R}$,
\item Every system of parameters for $R$ is a weak $M$-sequence.
\item $M$ is Gorenstein flat when viewed as an $R\ltimes \Omega$-module, that is $Z(M)\in\t{GFlat}(R\ltimes\Omega)$.
\item For every $R$-sequence $\bm{x}$, $\t{Tor}_{1}(R/(\bm{x}),M)=0$, where $(\bm{x})$ is the ideal of $R$ generated by $\bm{x}$.
\item For every $R$-sequence $\bm{x}$ and $i\geq 1$, $\t{Tor}_{i}(R/(\bm{x}),M)=0$, where $(\bm{x})$ is the ideal of $R$ generated by $\bm{x}$.

\end{enumerate}
\end{thm}
\ni
Recall that a \ti{system of parameters} for $R$ is a set of $d$ elements $\bm{y}=y_{1},\cdots,y_{d}$ in $R$ such that $\sqrt{\bm{y}}=\mf{m}$. These are precisely the maximal $R$-sequences.
\ni
Using Holm's result, we are now in a position to directly compare the categories $\hcm{R}$ and $\lcm{R}$, at least in the situation where $R$ admits a canonical module.

\begin{thm}\label{sep}
Let $(R,\mf{m},k)$ be a Cohen-Macaulay ring. If $\t{dim}\,R=1$ then $\lcm{R}$ and $\hcm{R}$ coincide. If $R$ also admits a canonical module, then $\lcm{R}=\hcm{R}$ if and only if $\t{dim}\,R=1$.
\end{thm}

\begin{proof}
Assume $\t{dim }R=1$ and $M\in\hcm{R}$. We can write $M=\rlim M_{i}$ where each $M_{i}$ is a finitely generated submodule of $M$, but since $\hcm{R}$ is closed under submodules, each $M_{i}\in\cm{R}$, so $M\in\lcm{R}$, which proves the first claim. 
\\
Now assume that $\t{dim }R>1$ and $R$ admits a canonical module. Let $x\in \mf{m}$ be a regular element and $\mf{p}$ be a minimal prime of the principal ideal $(x)$. By \cite[Cor. 11.17]{am}, $\mf{p}$ is a height one prime, and since $\t{dim }R>1$ it follows that $\mf{p}\neq \mf{m}$. Consider the indecomposable injective module $E:=E(R/\mf{p})$. We claim that $E\in\hcm{R}$ but not in $\lcm{R}$. Since $E$ is injective, we have $\t{Ext}_{R}^{i}(k,E)=0$ for all $i>0$, so in order for $E\in\hcm{R}$, it suffices to show $\Hom_{R}(k,E)=0$. Notice that any morphism $k\ra E$ will factor through $E(k)$, but since $\mf{p}\neq \mf{m}$ we have $\Hom_{R}(E(k),E)=0$, so it must be that $\Hom_{R}(k,E)=0$ and $E\in\hcm{R}$. To show that $E$ is not in $\lcm{R}$, we find a system of parameters that is not a weak $E$-sequence. Since $x$ is a regular element of $R$, we may extend it to a system of parameters $\bm{x}$. We claim $\bm{x}$ is not a weak $E$-sequence. Indeed, since $E=E(R/\mf{p})$, the unique associated prime of $E$ is $\mf{p}$, so there is an $e\in E$ with $\mf{p}=\t{Ann}(e)$. Since, by construction, we have $x\in\mf{p}$, it follows that $xe=0$, so $x$ is not a regular element on $E$, and therefore $\bm{x}$ is a system of parameters that is not a weak $E$-sequence.

\end{proof}
\ni
Therefore whenever $\t{dim }R>1$ and $R$ admits a canonical module, we have two different definable subcategories of $\Mod{R}$, both of whose finitely generated modules coincide precisely with $\cm{R}$. We can then consider which properties of $\cm{R}$ are reflected in these larger categories in an attempt to further differentiate between them. As stated, the canonical duality $\Hom_{R}(-,\Omega)$ is vital in understanding $\cm{R}$. We will now see how this functor behaves on $\lcm{R}$ and $\hcm{R}$.

\begin{prop} \label{hom}
Let $R$ be a complete Cohen-Macaulay ring. Then the functor $\Hom_{R}(-,\Omega)$ is an endofunctor on $\lcm{R}$ and $\Omega$ is an injective object in the category.
\end{prop}

\begin{proof}
Let $M\in\lcm{R}$. We show that if $\bm{x}$ is an $R$-sequence, then $\t{Tor}_{1}(R/(\bm{x}),\Hom(M,\Omega))=0$. Since $R$ is complete, by local duality we have 
\begin{equation}\label{star}
\t{Tor}_{1}(R/(\bm{x}),\Hom(M,\Omega))\simeq\t{Tor}_{1}(R/(\bm{x}),\Hom(H_{\mf{m}}^{d}(M),E(k)))\simeq \t{Ext}^{1}(R/(\bm{x}),H_{\mf{m}}^{d}(M))^{\vee}
\end{equation}
and since $E(k)$ is an injective cogenerator, we see it is enough to show that $\t{Ext}_{R}^{1}(R/(\bm{x}),H_{\mf{m}}^{d}(M))=0$. Using the assumption $M\in\lcm{R}$, write $M=\rlim M_{i}$ where each $M_{i}\in\cm{R}$. As $R/(\bm{x})$ is finitely generated and $H_{\mf{m}}^{d}(-)$ commutes with direct limits, we have an isomorphism
$$\t{Ext}_{R}^{1}(R/(\bm{x}),H_{\mf{m}}^{d}(M))\simeq\rlim \t{Ext}_{R}^{1}(R/(\bm{x}),H_{\mf{m}}^{d}(M_{i})).$$
We know that $\Hom(M_{i},\Omega)\in\cm{R}$ for each $M_{i}$, and since $\cm{R}\subset\lcm{R}$ we must have $\t{Ext}_{R}^{1}(R/(\bm{x}),H_{\mf{m}}^{d}(M_{i}))=0$ by considering the isomorphisms in (\ref{star}). This shows the first result. Since $R$ is complete, all finitely generated modules are Matlis reflexive and therefore pure injective (see \cite[Prop. 5.3.7]{rha}). Therefore by {\cite[Lemma 3.3.4]{trlifaj}}, if $j \geq 1$ and $M\in\lcm{R}$, there is an isomorphism 
$$\t{Ext}_{R}^{j}(M,\Omega)\simeq \t{Ext}_{R}^{j}(\rlim M_{i},\Omega)\simeq \llim \t{Ext}_{R}^{j}(M_{i},\Omega).$$
We know $\t{Ext}_{R}^{j}(M_{i},\Omega)=0$ for all $j\geq 1$ as $\Omega$ is injective in $\cm{R}$, hence $\t{Ext}_{R}^{j}(M,\Omega)=0$. This shows the second claim.
\end{proof}

\ni
This is a weaker result than for the finitely generated case, since $\Omega$ is no longer an injective cogenerator. Indeed, if $M\in\lcm{R}$ has infinite Ext-depth, we have $\Hom_{R}(M,\Omega)\simeq H_{\mf{m}}^{d}(M)^{\vee} = 0$. We will consider the modules where this is not the case in due course. Let us now consider the corresponding result for $\hcm{R}$.

\begin{prop}\label{hom2}
Let $R$ be a complete Cohen-Macaulay ring of dimension $d$. Then $\Hom(-,\Omega)$ is an endofunctor on $\hcm{R}$ and $\Omega$ is an injective object in this category.
\end{prop}

\begin{proof}
Let $M\in\hcm{R}$. If $M$ has infinite Ext-depth then $\Hom_{R}(M,\Omega)=0$ by the above discussion. Therefore we may assume that $\t{E-dp }M=d = \t{dim }R$. Since $R$ is complete, we may use local duality \ref{gd} to show that $\t{Ext}_{R}^{i}(k,H_{\mf{m}}^{d}(M)^{\vee})=0$ for all $i<d$, and thus $M^{*}\in\hcm{R}$. By considering the Ext-Tor relations given in \ref{relations}, this is equivalent to showing that $\t{Tor}_{i}^{R}(k,H_{\mf{m}}^{d}(M))=0$ for all $i<d$. Since $d$ is the cohomological dimension of $R$, by \cite[6.1.10]{lc} there is a natural isomorphism $H_{\mf{m}}^{d}(-)\simeq -\otimes_{R}H_{\mf{m}}^{d}(R)$. Therefore 
$$\t{Tor}_{i}^{R}(k,H_{\mf{m}}^{d}(M))\simeq \t{Tor}_{i}^{R}(k, M\otimes_{R}H_{\mf{m}}^{d}(R)),$$ and we aim to show this is zero for all $i<d$. In order to do this, we will use a spectral sequence argument; in particular, we show that, with the given assumption on $M$, there is a Grothendieck spectral sequence
\begin{equation}\label{ss}
E_{2}^{p,q}=\t{Tor}_{p}^{R}(M,\t{Tor}_{q}^{R}(k, H_{\mf{m}}^{d}(R))) \implies \t{Tor}_{p+q}^{R}(k,M\otimes H_{\mf{m}}^{d}(R))
\end{equation}
By \cite[Thm. 10.59]{rot}, such a spectral sequence exists if $\t{Tor}_{i}^{R}(M,H_{\mf{m}}^{d}(R)\otimes P)=0$ for all $i\geq 1$ and every projective $R$-module $P$. Since $R$ is a local ring every projective module is free and as Tor commutes with direct sums, it suffices to show $\t{Tor}_{i}^{R}(M,H_{\mf{m}}^{d}(R))=0$ for all $i\geq 1$. Yet 
\begin{equation*}
\t{Tor}_{i}^{R}(M,H_{\mf{m}}^{d}(R))=0\iff \t{Tor}_{i}^{R}(M,H_{\mf{m}}^{d}(R))^{\vee}=0 \iff \t{Ext}_{R}^{i}(M, H_{\mf{m}}^{d}(R)^{\vee})=0.
\end{equation*}
But since $\Omega$ is finitely generated, it is Matlis reflexive, so $H_{\mf{m}}^{d}(R)^{\vee}\simeq \Omega$, and as $\t{Ext}_{R}^{i}(M,\Omega)=0$ for all $i\geq 0$ by \ref{gd} and assumption on the Ext-depth of $M$, it follows $\t{Tor}_{i}^{R}(M,H_{\mf{m}}^{d}(R))=0$.
Therefore, for any $M\in\hcm{R}$ with $\t{E-dp }M=d$ the spectral sequence (\ref{ss}) exists. However, $\t{Tor}_{i}^{R}(k,H_{\mf{m}}^{d}(R))=0$ for all $i\neq d$, since
$$\t{Tor}_{i}^{R}(k, H_{\mf{m}}^{d}(R))^{\vee}\simeq \t{Ext}_{R}^{i}(k, H_{\mf{m}}^{d}(R)^{\vee})\simeq \t{Ext}_{R}^{i}(k,\Omega),$$
and $\t{Ext}_{R}^{i}(k,\Omega)=0$ for all $i\neq d$ by virtue of $\Omega$ being the canonical module. Consequently, $E_{2}^{p,q}$ is zero whenever $q\neq d$, so the spectral sequence collapses on the first page, giving isomorphisms
$$\t{Tor}_{p}^{R}(M,\t{Tor}_{q}^{R}(k, H_{\mf{m}}^{d}(R))) \simeq \t{Tor}_{p+q}^{R}(k,M\otimes H_{\mf{m}}^{d}(R))$$
In particular, as $\t{Tor}_{p}^{R}(M,\t{Tor}_{q}^{R}(k, H_{\mf{m}}^{d}(R)))=0$ for $q\neq d$, we see $\t{Tor}_{i}^{R}(k, M\otimes H_{\mf{m}}^{d}(R))=0$ for all $i<d$. That $\Omega$ is injective follows immediately from \ref{gd}.
\end{proof}

\begin{cor}\label{edual}
If $\t{E-dp}(M)=d$, then $\t{E-dp}(M^{*})=d$.
\end{cor}
\begin{proof}
From the above proof, there is an isomorphism
$$\t{Tor}_{p}(M,\t{Tor}_{q}(k, H_{\mf{m}}^{d}(R))) \simeq \t{Tor}_{p+q}(k, H_{\mf{m}}^{d}(M))$$
and the left hand side is zero whenever $q\neq d$. If we show
$M\otimes \t{Tor}_{d}(k,H_{\mf{m}}^{d}(R))\neq 0$, then $\t{Tor}_{d}^{R}(k, H_{\mf{m}}^{d}(M))\neq 0$, and therefore $\t{Ext}_{R}^{d}(k, \Hom_{R}(M,\Omega))\neq 0$ by \ref{relations}. One can use local duality and Ext-Tor relations to show that $\t{Tor}_{d}(k,H_{\mf{m}}^{d}(R))\simeq k$, and therefore $\t{Tor}_{d}(k,H_{\mf{m}}^{d}(R))\simeq k$ as $k$ is a simple $R$-module. Since $\t{E-dp}(M)=d<\infty$, we have $\t{T-codp}(M)=0$. This shows $M\otimes \t{Tor}_{d}(k,H_{\mf{m}}^{d}(R))\neq 0$, proving the claim.
\end{proof}

\begin{rmk}
In \cite{hc} a construction is given which shows that the functor $\Hom_{R}(-,\Omega)$ restricts to a duality on a certain subcategory of $\hcm{R}$, in fact a considerably more general result is given. We will restrict to the case when $R$ is a $d$-dimensional Cohen-Macaulay ring with a canonical module $\Omega$. Following \cite[6.1]{hc}, define the $\mf{m}$-adic completion functor to be
$$\Lambda^{\mf{m}}(-):=\llim_{t}\left( R/\mf{m}^{t}\otimes_{R}-\right)$$
and its $i$-th left derived functors is $H_{i}^{\mf{m}}(-)$ and called the \ti{$i$-th local homology} functor. For any $R$-module $M$, there is a canonical map $\psi_{M}:M\ra H_{0}^{\mf{m}}(M)$. Celikbas and Holm define the class of \ti{relative Cohen-Macaulay modules of cohomological dimension $d$ with respect to $\mf{m}$}, denoted $\t{CM}_{\mf{m}}^{d}(R)$, to consist of all $R$-modules $M$ such that
\begin{enumerate}
\item $H_{\mf{m}}^{i}(M)=0$ for all $i\neq d$;
\item The canonical map $\psi_{M}$ is an isomorphism;
\item if $\phi:M\ra M^{\vee\vee}$ is the canonical embedding into the Matlis double dual, then $H_{\mf{m}}^{i}(\t{coker}\,\phi)=0$ for all $i\in\mbb{Z}$. 
\end{enumerate}
\ni
It is shown in \cite[Thm. 6.16]{hc} that the functor $\Hom_{R}(-,\Omega)$ is a duality on the class $\t{CM}_{\mf{m}}^{d}(R)$, which is clearly a subcategory of $\hcm{R}$ from the first condition. \qed
\end{rmk}

\begin{rmk}
There is a shorter alternative proof for both \ref{hom2} and \ref{edual} without using spectral sequences, and the author is grateful to the referee for suggesting it. 
\begin{proof}
By local duality, one can see that $M\in\hcm{R}$ if and only if $\t{Ext}_{R}^{i}(M,\Omega)=0$ for $i\geq 0$, meaning that $\tb{R}\Hom_{R}(M,\Omega)\simeq \Hom_{R}(M,\Omega)$ in $\tb{D}(R)$. There is an equality
$$\t{E-dp}\,\Hom_{R}(M,\Omega) = \t{T-codp}\,M + \t{E-dp}\,\Omega = \t{T-codp}\,M+d$$ 
by \cite[Thm. 2.4]{yass}. Consequently $\t{E-dp}\,\Hom_{R}(M,\Omega)\geq d$ so is in $\hcm{R}$. Moreover, if $\t{E-dp}\,M=d$, we know $\t{T-codp}\,M=0$ by \ref{strook}, hence $\t{E-dp}\,\Hom_{R}(M,\Omega)=d$.
\end{proof}
\end{rmk}
\ni
Given any definable subcategory $\mc{D}\subset \Mod{R}$, there is a \ti{dual definable subcategory}, which we denote $\mc{D}^{d}$. This is the definable subcategory defined by the property that for any module $M\in\mc{D}$ we have $M^{\vee}\in\mc{D}^{d}$. It is indeed the case that this is a duality, since $\mc{D}^{dd}\simeq \mc{D}$ (see \cite[Cor. 3.4.18]{psl}). We can use \ref{relations} to describe the dual definable subcategories of both $\hcm{R}$ and $\lcm{R}$. Over any $d$-dimensional Cohen-Macaulay ring, the dual definable subcategory of $\hcm{R}$ is 
$$\hcm{R}^{d}=\{M\in\Mod{R}:\t{Tor}_{i}^{R}(k,M) = 0 \t{ for all }i<d\},$$
while if $R$ has a canonical module the dual definable subcategory of $\lcm{R}$ is
$$\lcm{R}^{d} = \{M\in\Mod{R}:\t{Ext}_{R}^{1}(R/(\bm{x}),M)=0 \t{ for all $R$-sequences $\bm{x}$}\}.$$
In the case that $R$ is a Gorenstein ring, and we can associate $\lcm{R}$ with the category of Gorenstein flat $R$-modules, $\lcm{R}^{d}$ is precisely the class of \ti{Gorenstein injective} $R$-modules. This follows from \cite[Thm. 10.3.8(7)]{rha} and \cite[Cor. 10.3.9]{rha}. Information about Gorenstein injective modules can be found in \cite[\S 10.1]{rha}.
\\
\\
In the proof of the above theorem, we showed that if $M\in\hcm{R}$, then $\t{Tor}_{i}^{R}(k, M\otimes_{R}H_{\mf{m}}^{d}(R)) = 0$ for all $i<d$. This shows that $-\otimes_{R}H_{\mf{m}}^{d}(R)$ gives us a functor $\hcm{R}\ra \hcm{R}^{d}$. However, this functor is very far from a duality. The following result helps to see this.

\begin{prop}
Let $(R,\mf{m},k)$ be a $d$-dimensional Cohen-Macaulay ring and $M$ an $R$-module with $\t{E-dp}(M)=d$. Then $H_{\mf{m}}^{i}(H_{\mf{m}}^{d}(M))\simeq H_{\mf{m}}^{i+d}(M)$ for all $i\geq 0$.
\end{prop}
\begin{proof}
Since $\Gamma_{\mf{m}}\circ\Gamma_{\mf{m}}$ and $\Gamma_{\mf{m}}(E)$ is injective for any injective $R$-module $E$, there is a Grothendieck spectral sequence
$$E_{2}^{p,q}=H_{\mf{m}}^{p}(H_{\mf{m}}^{q}(N))\implies H_{\mf{m}}^{n}(N)$$
for any $R$-module $N$. If $\t{E-dp}(M)=d$, then $H_{\mf{m}}^{i}(M)=0$ for all $i\neq d$, so this spectral sequence is nonzero only if $q=d$. Consequently this collapses on the first page giving isomorphisms $H_{\mf{m}}^{d+i}(M)\simeq H_{\mf{m}}^{i}(H_{\mf{m}}^{d}(M))$.
\end{proof}
\ni
Consequently, if $M\in\hcm{R}$, then $M\otimes_{R} H_{\mf{m}}^{d}(R)\otimes_{R} H_{\mf{m}}^{d}(R)\simeq H_{\mf{m}}^{d}(H_{\mf{m}}^{d}(M))\simeq H_{\mf{m}}^{2d}(M)=0$, showing that $-\otimes_{R}H_{\mf{m}}^{d}(R):\hcm{R}\ra \hcm{R}^{d}$ is not a duality.

\section[Ext-depth]{Ext-depth and $\rlim\cm{R}$}
\ni
Unless explicitly stated, for this section we assume that $(R,\mf{m},k)$ is a $d$-dimensional Cohen-Macaulay $R$-module admitting a canonical module. So far we have only really considered Ext-depth in relation to $\hcm{R}$, but it also plays an impact on $\lcm{R}$. Suppose $M\in\lcm{R}$ can be realised as $M=\rlim M_{i}$ with each $M_{i}\in \cm{R}$, so we have an isomorphism $\t{Ext}_{R}^{d}(k,M)\simeq \rlim \t{Ext}_{R}^{d}(k,M_{i})$, and this module can be zero, meaning that $\t{E-dp}(M)\in\{d,\infty\}$. One can obtain an example of an infinite Ext-depth module in $\lcm{R}$ quite easily - indeed, if $R$ is a Gorenstein ring, then $E(R)$, the injective hull of $R$, is a flat module, and it is clear that every flat module lies in $\lcm{R}$. We will let $\lcm{R}_{\infty}$ denote the subcategory of $\lcm{R}$ consisting of all infinite Ext-depth modules, and $\lcm{R}_{d}$ denote the subcategory consisting of all modules with Ext-depth $d$. We will first turn our attention to $\lcm{R}_{d}$, and to do so we recall the following definition due to Hochster.

\begin{defn}
An $R$-module $M$ is a \ti{balanced big Cohen-Macaulay module} if every system of parameters is a regular sequence. We will let $\t{bbCM}(R)$ denote the class of balanced big Cohen-Macaualy modules.
\end{defn}

\ni
By Holm's result \ref{holm} any balanced big Cohen-Macaualy module will be an element of $\lcm{R}$. Suppose $\bm{y}$ is a system of parameters for $R$. Since local cohomology is invariant under radical, there are isomorphisms of functors $H_{(\bm{y})}^{i}(-)\simeq H_{\mf{m}}^{i}(-)$ for each $i\geq 0$. In particular, if $M\in\lcm{R}$ has $\t{E-dp}(M)=d$, we see that $H_{(y)}^{d}(M)\neq 0$, so $\t{T-codp}(\bm{y},M)=0$ and thus $M/\bm{y}M\neq 0$. Since $M\in\lcm{R}$, we know that $\bm{y}$ is a weak $M$-sequence, and therefore we see that $\bm{y}$ is actually an $M$-sequence. Conversely, if $\t{E-dp}(M)=\infty$, it follows that $\t{T-codp}(\bm{y},M)=\infty$, so $R/(\bm{y})\otimes M=0$ so $\bm{y}$ is not a regular $M$-sequence. Therefore we have shown the following.

\begin{lem}
Let $R$ be a Cohen-Macaulay ring with a canonical module $\Omega$. Then the modules in $\lcm{R}$ of Ext-depth $d$ coincide with the category of balanced big Cohen-Macaulay modules, that is
$$\t{bbcm}{R}=\lcm{R}_{d}.$$
\end{lem}
\ni
This is not a definable subcategory of $\Mod{R}$, since it is not closed under direct limits, moreover, it is not closed under direct summands, nor does it contain the zero module. However, it is closed under direct sums and direct products. Since it contains the canonical module $\Omega$, we can consider how the dual acts on it.
\begin{prop}
Let $(R,\mf{m},k)$ be a $d$-dimensional complete Cohen-Macaulay ring with canonical module $\Omega$. If $M$ is a balanced big Cohen-Macaulay $R$-module, the following hold:
\begin{enumerate}
\item $M^{*}=\Hom(M,\Omega)$ is also a balanced big Cohen-Macaulay module.
\item $\Omega$ is an injective cogenerator in $\t{bbCM}(R)$.
\item If $M$ does not have a direct summand of infinite depth, then the canonical morphism $M\ra M^{**}$ is injective.
\end{enumerate}
\end{prop}

\begin{proof}
\leavevmode
\begin{enumerate}
\item We know that $M^{*}\in\lcm{R}$ so it suffices to show that $\t{E-dp}(M)=d$, but this is what \ref{edual} shows.
\item This follows from Grothendieck local duality.
\item The proof of this is essentially the same as the proof that $M\simeq M^{**}$ for $M\in\cm{R}$. Indeed, if $\bm{x}$ is an $R$-sequence, we may extend it to a system of parameters which is then an $M$-sequence as $M$ is a balanced big Cohen-Macaulay module. If $M$ does not have a direct summand of infinite Ext-depth, we can reduce to the case that $\t{dim }R=0$, as is done in [B.H 3.3.10]. In this situation, $\Omega\simeq E(k)$, and then $M\ra M^{\vee\vee}$ is injective.
\end{enumerate}
\end{proof}
\ni
Having considered $\lcm{R}_{d}$ we will now turn our attention to $\lcm{R}_{\infty}$. Notice that we can determine the Ext-depth of a module in $\lcm{R}$ by considering if the functor $k\otimes -$ vanishes on it: clearly $\t{E-dp}(M)=d$ if and only if $k\otimes M\neq 0$.
\begin{lem}
If $R$ is a Cohen-Macaulay ring admitting a canonical module, then $\lcm{R}_{\infty}$ is a definable subcategory of $\Mod{R}$.
\end{lem}
\begin{proof}
Let $X$ be the set of functors defining $\lcm{R}$, and $X'=X\cup\{k\otimes -\}$. This set of finitely presented functors determines $\lcm{R}_{\infty}$. 
\end{proof}
\ni
It is clear that there are no finitely generated modules in $\lcm{R}_{\infty}$ but it is fully contained in $\lcm{R}$ so is still completely determined by $\cm{R}$. For example, any flat module of infinite Ext-depth lies in $\lcm{R}_{\infty}$, and there are no shortage of such modules: if $R$ is a Gorenstein ring, for instance, and $F$ is an arbitrary flat module, then the injective hull of $F$ is also flat and is of infinite depth. A proof of this fact can be found at \cite[Theorem 9.3.3(3)]{rha}. It is clear that there is an inclusion of $\lcm{R}_{\infty}$ in $\lcm{R}$, and the nature of this inclusion is quite familiar.

\begin{lem}\label{serre}
Let $0\ra L\ra M\ra N\ra 0$ be a short exact sequence of $R$-modules with $L,M,N\in\lcm{R}$. Then $M\in\lcm{R}_{\infty}$ if and only if $L,N\in\lcm{R}_{\infty}$.
\end{lem}

\begin{proof}
It is clear that $\lcm{R}_{\infty}$ is extension closed. Suppose $0\ra L\ra M\ra N\ra 0$ is a short exact sequence of modules in $\lcm{R}$ such that $M\in\lcm{R}_{\infty}$. By the assumptions on $L,M$ and $N$ there is then an exact sequence
$$0\ra \t{Ext}_{R}^{d}(k,L)\ra \t{Ext}_{R}^{d}(k,M)\ra\t{Ext}_{R}^{d}(k,N)\ra \t{Ext}_{R}^{d+1}(k,L)\ra \t{Ext}_{R}^{d+1}(k,M).$$
Since we assumed that $\t{E-dp}(M)=\infty$, we have $\t{Ext}_{R}^{d}(k,M)=0=\t{Ext}_{R}^{d+1}(k,M)$, showing that $\t{Ext}_{R}^{d}(k,L)=0$ and thus $\t{E-dp}(L)=\infty$. It immediately follows that $\t{E-dp}(N)=\infty$, which shows the result. 
\end{proof}
\ni

\begin{exmp}
Let us consider the $A_{\infty}$ curve singularity $R=k[[x,y]]/(x^{2})$. This is a complete one-dimensional Gorenstein ring, so $\lcm{R}=\hcm{R}=\{M\in\Mod{R}:\Hom(k,M)=0\}$. Up to isomorphism, there are countably many Cohen-Macaulay $R$-modules, which were classified by Buchweitz, Greuel and Schreyer in \cite{bgs}. They are:
\begin{enumerate}
\item the ring, $R$;
\item the ideals $I_{j}:=(x,y^{j}),$ where $j\geq 1$;
\item the ideal $I_{\infty}:=xR$.
\end{enumerate}
Since $R$ is a complete local ring, each of these is an indecomposable pure-injective $R$-module. The remaining indecomposable pure injective $R$-modules in $\lcm{R}$ were classified by Puninski in \cite{pun}. They are
\begin{enumerate}
\item $Q=Q(R)$, the total quotient ring of $R$;
\item $\ol{R}$, the integral closure of $R$ in $Q$;
\item the Laurent series $L:=k((y))$, viewed as an $R$-module through the morphism $R\ra R/(x)$.
\end{enumerate}
\ni
Let us now determine the Ext-depth of each of these indecomposable pure-injectives.
\begin{itemize}
\item Let us start with $Q$. Since $R$ is a Gorenstein ring, $Q$ is an injective $R$-module by \cite[9.3.3]{rha} and therefore $\t{Ext}_{R}^{i}(k,Q)=0$. Consequently $\t{E-dp}(Q)=\infty$.
\item Let us now consider $L:=k((y))$. The quotient $R\ra k[[y]]$ sends the maximal ideal $\mf{m}=(x,y)$ of $R$ to the maximal ideal $(y)$ of $k[[y]]$. The independence theorem of local cohomology shows that for each $i$ is an isomorphism $H_{\mf{m}}^{i}(L)\simeq H_{(y)}^{i}(L)\vert_{R}$, where on the right hand side we view $L$ as a $k[[y]]$-module. We know $H_{\mf{m}}^{0}(L)=0$, and therefore $H_{(y)}^{0}(L)=0$ as a $k[[y]]$-module, since the ring homomorphism is just factoring by $x$. But we also know that $k((y))$ is an injective $k[[y]]$ module for the same reason as in the case of $Q$, and therefore $H_{(y)}^{i}(k((y)))=0$ for all $i$. It follows from the independence theorem that $\t{E-dp}(L)=\infty$.
\item Lastly we consider $\ol{R}$. We show that $k\otimes_{R} \ol{R}\neq 0$, hence $\t{E-dp}(\ol{R})=1$. In \cite[Remark 2.1]{pun}, it is shown that $\ol{R}\supset y\ol{R}$ and $y\ol{R}$ is maximal. Thus $\ol{R}/y\ol{R}\simeq k$, as $k$ is the unique simple module. Since the sequence $\ol{R}\ra \ol{R}/y\ol{R}\ra 0$ is exact, so is $\ol{R}\otimes_{R}k\ra \ol{R}/y\ol{R}\otimes_{R}k\ra 0$, but $\ol{R}/y\ol{R}\otimes_{R}k\neq 0$, which means that $\ol{R}\otimes k \neq 0$, which is what we wanted to show.
\end{itemize}
\end{exmp}
\ni
If one defines $\hcm{R}_{d}$ and $\hcm{R}_{\infty}$ in the obvious way, several of the results in this section and their proofs can be easily adapted to $\hcm{R}$. We collate these here. 
\begin{prop}
Let $R$ be a Cohen-Macaulay ring of dimension $d$. 
\begin{enumerate}
\item 
If $0\ra L\ra M\ra N\ra 0$ is a short exact sequence of $R$-modules in $\hcm{R}$, then $M\in\hcm{R}_{\infty}$ if and only if $L,N\in\hcm{R}_{\infty}$.

\item $\hcm{R}_{\infty}$ is a definable subcategory of $\Mod{R}$ containing no finitely generated $R$-modules.

\item If $R$ admits a canonical module, then $\Hom_{R}(-,\Omega)$ is an endofunctor on $\hcm{R}_{d}$ and $\Omega$ is injective in this category.
\end{enumerate}
\end{prop}

\section[]{Some categorical properties of $\lcm{R}$ and $\hcm{R}$}
\ni
In this section, we will look at some of the categorical properties of $\lcm{R}$ and $\hcm{R}$. Unless stated otherwise, $(R,\mf{m},k)$ will be a Cohen-Macaulay ring. Since many of the properties we consider will not require a canonical module, we will explicitly state when we are assuming $R$ admits one.

\begin{defn}
A subclass of $R$-modules $\mc{G}\subset\Mod{R}$ is \ti{preenveloping} if for every $M\in\Mod{R}$ there is a morphism $\phi:M\ra G$ with $\mc{G}$ in $G$ such that for every morphism $\psi:M\ra G'$ with $G'\in \mc{G}$ there is a $\alpha:G\ra G'$ such that $\psi=\alpha\circ\phi$. 
$$
\begin{tikzcd}
M \arrow[r,"\psi"] \arrow[dr,swap,"\phi"] & G' \\
& G \arrow[u,swap,"\alpha"]
\end{tikzcd}
$$
We say that $\mc{G}$ is \ti{enveloping}, if whenever $\psi=\phi$ in the above diagram, then $\alpha\in\t{Aut}(G)$.
\end{defn}
\ni
The dual notions are \ti{precovering} and \ti{covering} respectively. \\
\\
For a class of $R$-modules $\mc{A}$, we define
$$\mc{A}^{\perp}=\{M\in\Mod{R}:\t{Ext}_{R}^{1}(A,M)=0 \t{ for all }A\in\mc{A}\};$$
$$^{\perp}\mc{A}=\{M\in\Mod{R}:\t{Ext}_{R}^{1}(M,A)=0 \t{ for all }A\in\mc{A}\}.$$
We say that a $\mc{G}$-(pre)envelope $\phi:M\ra G$ is \ti{special} if there is a short exact sequence
$$
\begin{tikzcd}
0 \arrow[r] & M \arrow[r, "\phi"] & G \arrow[r] & X \arrow[r] & 0
\end{tikzcd}
$$
such that $X\in\,^{\perp}\mc{G}$, while a $\mc{F}$-precover $\gamma:F\ra M$ is \ti{special} if there is a short exact sequence
$$
\begin{tikzcd}
0  \arrow[r] & Y \arrow[r] & F \arrow[r,"\gamma"] &M \arrow[r] & 0
\end{tikzcd}
$$
with $Y\in \mc{F}^{\perp}.$
The following result, known as Wakamatsu's lemma, enables us to relate envelopes with special (pre)envelopes.

\begin{lem}{\cite[2.1.13]{trlifaj}}
Let $M\in \Mod{R}$ and $\mc{C}$ an extension closed class of $R$-modules.
\begin{enumerate}
\item Let $f:M\ra C$ be an injective $\mc{C}$-envelope of $M$. Then $f$ is special.
\item Let $g:C\ra M$ be a surjective $\mc{C}$-cover of $M$. Then $g$ is special.
\end{enumerate}
\end{lem}
\begin{defn}
A pair $\mf{C}=(\mc{F},\mc{G})$ of classes of $R$-modules is called a \ti{cotorsion theory} if $\mc{F}^{\perp}=\mc{G}$ and $^{\perp}\mc{G}=\mc{F}$.
\end{defn}
\ni
If $\mf{C}=(\mc{F},\mc{G})$ is a cotorsion pair, we say it is
\begin{enumerate}
\item \ti{hereditary} if $\mc{F}$ is closed under kernels of epimorphisms.
\item \ti{perfect} if $\mc{F}$ is covering and $\mc{G}$ is enveloping.
\item \ti{closed} if $\mc{F}$ is closed under direct limits.
\end{enumerate}

\ni
\\
We can relate some of these notions to the classes of modules we have previously seen: the following result is due to Holm.
\begin{thm}{\cite[Theorem D]{holm}}\label{lcmpair}
Let $R$ be a Cohen-Macaulay ring admitting a canonical module. Then the pair $(\lcm{R},\lcm{R}^{\perp})$ is a perfect hereditary cotorsion pair on $\Mod{R}$. Moreover, $\lcm{R}$ is preenveloping in $\Mod{R}$.
\end{thm}

\ni
In fact, the conclusion the $\lcm{R}$ is preenveloping follows immediately from the fact that $\lcm{R}$ is the definable subcategory of $\Mod{R}$ generated by $\cm{R}$, since all definable subcategories of $\Mod{R}$ are preenveloping in $\Mod{R}$. 
\ni
We now show the corresponding result for $\hcm{R}$.

\begin{thm}\label{cotorsion}
Let $R$ be any Cohen-Macaulay ring. The pair $$(\hcm{R},\hcm{R}^{\perp})$$ is a perfect hereditary closed cotorsion pair in $\Mod{R}$.
\end{thm}

\ni
Before proving this, recall that a class $\mc{F}\subset \Mod{R}$ is a \ti{Kaplansky class} if there is a cardinal $\lambda$ such that for every module $M\in\mc{F}$ and $x\in M$ there is a submodule $N\subset M$ such that $x\in N\subset M$, both $N,M/N\in\mc{F}$ and $\vert N\vert \leq \lambda$.

\begin{proof}[Proof of \ref{cotorsion}]
It is clear that $\hcm{R}$ is extension closed and contains all the projective $R$-modules. Therefore by {\cite[Theorem 2.8]{kap}}, the result holds if $\hcm{R}$ is a Kaplansky class. If $\lambda> \t{card}(R)+\aleph_{0}$, then for every $M\in\hcm{R}$ and $x \in M$, there is a pure submodule $N$ of $M$ such that $x\in N\subset M$ with $\t{card}(N)\leq\lambda$. Since $\hcm{R}$ is definable, it follows that both $N$ and $M/N$ are both in $\hcm{R}$; in particular, $\hcm{R}$ is a Kaplansky class. 
\end{proof}
\ni
The above theorem was already known in \cite{duality}. Indeed, combining \cite[1.9]{duality} and \cite[Thm. 3.1]{duality} also yields the result. In fact, the proof of the above, combined with the results, enables a partial extension of \ref{lcmpair}. To show this, we prove a much more general result.

\begin{prop}
Let $R$ be a coherent ring and $\mc{C}$ a class of finitely presented right $R$-modules containing $R$. Then $(\rlim \mc{C},(\rlim\mc{C})^{\perp})$ is a perfect cotorsion pair in $\t{Mod-}R$.
\end{prop}
\begin{proof}
By \cite[2.3]{duality} and \cite[Thm. 3.1]{duality}, the class $\rlim \mc{C}$ is closed under pure submodules and pure quotients. The proof of \ref{cotorsion} only required closure of pure submodules and pure quotients, so it follows that $\rlim\cm{R}$ is a Kaplansky class containing all the projective $R$-modules. Since $R$ is coherent, we may apply \cite[Thm. 4.5.6]{trlifaj} which shows $\rlim \cm{C} = \,^{\top}(\mc{C}^{\top})$, where 
$$\mc{C}^{\top} = \{M\in R\t{-Mod}:\t{Tor}_{1}^{R}(C,M)=0 \t{ for all }C\in\mc{C}\},$$ 
and we similarly define $\,^{\top}(\mc{C}^{\top})$. In particular, the class $\rlim \mc{C}$ is an extension closed Kaplansky class containing the projective $R$-modules. We therefore may apply {\cite[Theorem 2.8]{kap}} to deduce the result.
\end{proof}
\begin{cor}
If $R$ is any Cohen-Macaulay ring, then $(\lcm{R},\lcm{R}^{\perp})$ is a perfect cotorsion pair.
\qed
\end{cor}
\ni
Unlike in \ref{lcmpair}, we cannot deduce that the above cotorsion pair is hereditary over an arbitrary Cohen-Macaulay ring.
\\
\\
Every definable subcategory of $\Mod{R}$ is covering and preenveloping, but \ref{cotorsion} enables us to take special $\hcm{R}$-precovers and special $\hcm{R}^{\perp}$-preenvelopes. Since the class is special precovering, we can take minimal left resolutions of any $R$-module in the obvious way.

\begin{defn}
Let $M$ be an $R$-module, we will let $\t{L-dim}(M)$ denote the minimal length of a left $\hcm{R}$-resolution.
\end{defn}

\begin{lem}
Let $M$ be an $R$-module. Then $\t{E-dp}(M)+L\t{-dim}(M)\geq \t{dim }R$.
\end{lem}
\begin{proof}
Let us assume that $L\t{-dim}(M)=n$ with $0<n\leq d$, else there is nothing to prove. Then there is an exact sequence
$$\mc{L}: 0 \ra C_{n}\xra{\phi_{n}} C_{n-1}\xra{\phi_{n-1}}\cdots \ra C_{1}\xra{\phi_{1}}C_{0}\xra{\phi_{0}}M\ra 0$$
with each $C_{i}\in\hcm{R}$. Since $M \simeq \t{Im }\phi_{0} \simeq C_{0}/\t{Ker }\phi_{0} \simeq \t{Coker }\phi_{1}$, there is a short exact sequence $0 \ra \t{Im }\phi_{1}\ra C_{0} \ra \t{Coker }\phi_{1}$. Moreover, $\t{Im }\phi_{1} \simeq C_{1}/\t{Ker }\phi_{1}$. Yet by exactness, $\t{Ker }\phi_{1} \simeq \t{Im }\phi_{2} \simeq C_{2}/\t{Ker }\phi_{2} \simeq \t{Coker }\phi_{3}$, so one obtains a second short exact sequence $0\ra \t{Coker }\phi_{3} \ra C_{1} \ra \t{Im }\phi_{1}\ra 0$. Continuing this process, one can decompose $\mc{L}$ into a collection of short exact sequences 
$$ 0 \ra \t{Im }\phi_{2k+1}\ra C_{2k}\ra \t{Coker }\phi_{2k+1}\ra 0$$
$$0 \ra \t{Coker }\phi_{2k+3} \ra C_{2k+1} \ra \t{Im }\phi_{2k+1} \ra 0,$$
where $M = \t{Coker }\phi_{1}$. If $n=2m$ is even, then $\phi_{n+1}$ is the zero map and consequently there is an isomorphism $C_{2m}\simeq \t{Coker }\phi_{2m+1}$. There are therefore $n$ exact sequences to consider. By repeated application of the depth lemma (\ref{dl}), we see $\t{E-depth}(\t{Im }\phi_{n-l})\geq d-l$ for all $l<n$, and so $\t{E-depth}(\t{Im }\phi_{1})\geq d+1-n$. By applying the depth lemma one final time to the exact sequence
$$0 \ra \t{Im }\phi_{1}\ra C_{0}\ra M\ra 0,$$
we see that $\t{E-depth}(M)\geq d-n$, hence $\t{E-depth}(M)+n\geq d$. In the case when $n$ is odd, an almost identical argument yields the same inequality.
\end{proof}

\begin{cor}
Let $M\in\Mod{R}$ be of finite Ext-depth, then $L\t{-dim}(M)\geq \t{T-codp}(M)$.
\end{cor}

\begin{proof}
By the above, we have an inequality $\t{E-dp}(M)+L\t{-dim}(M)\geq \t{dim }R$, but we also know that $\t{E-dp}(M)+\t{T-codp}\leq \t{dim }R$ by \ref{strook}. Combing these inequalities gives the result.
\end{proof}
\ni
The categories $\hcm{R}$ and $\hcm{R}_{\infty}$ also have some interesting properties in their own right, independently of their relationship to $\cm{R}$.

\begin{prop}
Let $R$ be a Cohen-Macauly ring of dimension at least one. Then $\hcm{R}$ and $\hcm{R}_{\infty}$ are closed under injective hulls. Consequently both categories have enough injectives.
\end{prop}

\begin{proof}
If $\mf{p}\neq \mf{m}$ is a prime ideal, then $\Hom_{R}(k,E(R/\mf{p}))=0$. Indeed, any morphism $k\ra E(R/\mf{p})$ factors through $E(k)$, but $\Hom_{R}(E(k),E(R/\mf{p}))=0$ by the assumption that $\mf{p}\neq \mf{m}$. By Matlis's results on injective modules, if $M$ is an $R$-module its injective hull is of the form
$$E(M)\simeq \bigoplus_{\mf{p}\in\spec{R}}E(R/\mf{p})^{(X_{\mf{p}})},$$
where $\t{card}(X_{\mf{p}}) = \t{dim}_{k(\mf{p})}\Hom_{R}(R/\mf{p},M)_{\mf{p}}$. In particular, if $\t{E-dp}(M)>0$ then $X_{\mf{m}}=0$ so $\Hom_{R}(k,E(M))=0$ by the above reasoning. Consequently $E(M)$ has infinite Ext-depth. Therefore $E(M)\in \hcm{R}_{\infty}$ for any $M$ with $\t{E-dp}(M)>0$, so $\hcm{R}$ and $\hcm{R}_{\infty}$ are closed under injective hulls.
\end{proof}

\ni
Let us now turn our attention to inverse limits. Firstly, let's see how the inverse limit closure relates to $\lcm{R}$.

\begin{lem}
Let $R$ be a complete Cohen-Macaulay ring. Then the inverse limit closure of $\cm{R}$, denoted $\llim\cm{R}$, is a subcategory of $\lcm{R}$.
\end{lem}
\begin{proof}
Let $(M_{i}, f_{i}^{j})_{I}$ be an inverse system of modules in $\cm{R}$ with inverse limit $M$. Since each $M_{i}$ is finitely presented is Matlis reflexive, so the inverse system can be realised as the dual of a directed system $(N_{i},g_{i}^{j})_{I}$ where $N_{i}=M_{i}^{\vee}$ and likewise for $g_{i}^{j}$. Since $\lcm{R}$ consists of all modules vanishing on the set $\{\t{Tor}_{1}(R/(\bm{x}),-): \bm{x} \ti{ is an $R$-sequence}\}$, we may apply \ref{relations} to show that its dual definable category consists of all modules vanishing on the set of functors 
$$X=\{\t{Ext}_{R}^{1}(R/(\bm{x}),-):\bm{x} \t{ is an $R$-sequence}\}.$$
Since each $N_{i}$ is in this dual definable subcategory, so is the directed limit of the system $(N_{i},g_{i}^{j})_{I}$, which we denote by $N$. Then $\Hom_{R}(N,E)\in\lcm{R}$ by definition of the dual definable category, but
$$\Hom_{R}(N,E) = \Hom_{R}(\rlim N_{i},E)\simeq \llim \Hom_{R}(N_{i},E)\llim M_{i} = M.$$
Consequently $M\in \lcm{R}$, which shows the claim.
\end{proof}
\ni
In particular, we see that $\rlim \llim\cm{R}\subseteq \lcm{R}$. One may wonder if it is possible to swap the direct and inverse limits and reach the same conclusion, namely that $\lcm{R}$ is closed under inverse limits. In general, definable subcategories are not closed under inverse limits - for instance, over a Noetherian ring the category of injective $R$-modules is definable, but its inverse limit closure is the entire module category, see \cite{bergman}. 
We now show that, with an assumption on Krull dimension, $\hcm{R}$ is never closed under inverse limits.

\begin{lem}
Let $\t{dim }R\geq 3$ be a Cohen-Macaulay ring. Then $\hcm{R}$ is not closed under inverse limits.
\end{lem}

\begin{proof}
Since $\t{dim }R\geq 3$, we can choose an $R$-module $M$ such that $2\leq \t{E-dp}(M)<\t{dim }R$. We will show that $M$ can be realised as an inverse system of $R$-modules in $\hcm{R}$. Consider the start of a minimal injective resolution of $M$
\begin{equation}\label{mininj}
0 \ra M\ra E(M)\ra E^{1}
\end{equation}
where $E^{1}$ is the injective hull of $\t{coker}(M\ra E(M))$. By the choice of $M$, it is clear that $E(M)\in\hcm{R}$. Applying the depth lemma (\ref{dl}) to the short exact sequence $0\ra M\ra E(M)\ra E(M)/M\ra 0$ shows that $\t{E-dp}(E(M)/M)\geq 1$ so $E^{1}\in\hcm{R}$. Consequently (\ref{mininj}) is a short exact sequence in $\hcm{R}$. Since $\hcm{R}$ is closed under direct sums, one can apply \cite[Cor. 11]{bergman}, which shows that $M$ can be realised as an inverse limit of a countable inverse system of modules in $\hcm{R}$.
\end{proof}
\ni
In fact, one can draw a more general conclusion from Bergman's corollary - namely that any class that is both closed under direct summands and is special preenveloping is not closed under inverse limits. Clearly any special preenveloping definable subcategory satisfies this property. Since $E(k)$ is not a member of $\hcm{R}$ for any Cohen-Macaulay ring $R$, we cannot say that $\hcm{R}$ is special preenveloping in $\Mod{R}$. However, it does contain sufficiently many injective $R$-modules to start to form injective resolutions, from which one can apply Bergman's result, as done in the above proof.

\begin{exmp}
Let $R= k[[x,y,z,w]]$ be a four-dimensional regular local ring. Then the module $R/(x)$ has Ext-depth equal to three and is not Cohen-Macaulay as an $R$-module. Consequently we can use the above result to obtain $R/(x)$ as an inverse limit of modules in $\hcm{R}$.
\end{exmp}

\ni
There are, however, certain inverse systems in both $\lcm{R}$ and $\hcm{R}$ whose inverse limits remain in their respective category.
\begin{defn}{\cite[3.5]{trlifaj}}
A sequence of $R$-modules and homomorphisms of the form
$$ 
\mc{T}: \cdots\ra T_{i+1}\xra{t_{i}} T_{i} \xra{t_{i-1}}\cdots\xra{t_{1}}T_{1}\xra{t_{0}}T_{0}$$
is called a \ti{tower}.
\end{defn}
\ni
Clearly given a tower $\mc{T}$ one can form an inverse system $\mc{T}'=\{(T_{i},\phi_{i}^{j})\}_{\mbb{N}}$, where $\phi_{i}^{i}=\t{Id}_{T_{i}}$ and for any $i<j$ the morphism $\phi_{i}^{j}:T_{j}\ra T_{i}$ is given by the composition $\phi_{i}^{j}=t_{i}\circ \cdots \circ t_{j-1}$. Following the construction given in \cite[§3.1]{trlifaj}, given any tower $\mc{T}$ there is a map 
$$\Phi_{\mc{T}}:\prod_{i\in\mbb{N}} T_{i}\ra \prod_{i\in\mbb{N}}T_{i}$$
whose kernel is the inverse limit of the inverse system $\mc{T}'$. In particular, if $\t{Coker }\Phi_{\mc{T}}=0$, then there is a short exact sequence of $R$-modules
\begin{equation}\label{inlim}
0\ra \llim T_{i} \ra \prod_{i\in\mbb{N}}T_{i} \xra{\Phi_{\mc{T}}} \prod_{i\in\mbb{N}}T_{i} \ra 0.
\end{equation}
\ni
In order to give a situation when $\t{Coker }\Phi_{\mc{T}}=0$, we recall the Mittag-Leffler condition.
\begin{defn}\label{ml}
Let $(M_{i},f_{ij})_{I}$ be an inverse system of $R$-modules. We say the system satisfies the \ti{Mittag-Leffler condition} if for any $i\in I$ there is a $j\geq i$ such that for any $k\geq j$ we have $\t{im }f_{ik} = \t{im }f_{ij}$.
\end{defn}
\begin{lem}{\cite[3.6]{trlifaj}}
If $\mc{T}'$ satisfies the Mittag-Leffler condition, then $\t{Coker }\Phi_{\mc{T}}=0$.
\end{lem}
\ni
We can now prove the following result without much difficulty.
\begin{lem}
Let $\mc{T}$ be a tower in $\lcm{R}$ (resp. $\hcm{R}$). If the associated inverse system $\mc{T}'$ satisfies the Mittag-Leffler condition, then $\llim T_{i} \in\lcm{R}$ (resp. $\hcm{R}$).
\end{lem}
\begin{proof}
We will prove the result for $\lcm{R}$. By our assumptions, the sequence (\ref{inlim}) is exact. Since $\lcm{R}$ is definable, $\prod_{i\in\mbb{N}}T_{i}$ is also in $\lcm{R}$. If $\bm{x}$ is an $R$-sequence, applying the functor $R/(\bm{x})\otimes -$ to (\ref{inlim}) shows that $\t{Tor}_{1}(R/(\bm{x}), \llim T_{i})=0$, showing $\llim T_{i}\in\lcm{R}$ by \ref{holm}.
\end{proof}
\ni
Related to the Mittag-leffler conditions on an inverse system of modules is the notion of a Mittag-Leffler module.

\begin{defn}
Let $\mc{Q}$ be a class of $R$-modules. We say that an $R$-module $M$ is $\ti{$\mc{Q}$-Mittag-Leffler}$ if, for every collection $\{N_{i}\}_{i\in I}$ in $\mc{Q}$, the canonical morphism
$$M\otimes\prod_{i\in I} N_{i}\ra \prod_{i}(M\otimes N_{i})$$
is injective. If $\mc{Q}=\Mod{R}$, we say that an $R$-module is \ti{Mittag-Leffler}.
\end{defn}
\ni
$\mc{Q}$-Mittag-Leffler modules were considered from a Model theoretic perspective by P. Rothmaler in \cite{Rothmaler} and from a more algebraic viewpoint by D. Herbera and L. Angeleri-H\"{u}gel and in \cite{ahh}.
\begin{defn}
\begin{enumerate}
\item
Let $B$ be an $R$-module. A directed system $(M_{i},f_{ij})_{I}$ of $R$-modules is said to be \ti{$B$-stationary} if the inverse system $(\Hom_{R}(M_{i},B), \Hom_{R}(f_{ij},B))_{I}$ satisfies the Mittag-Leffler condition \ref{ml}. 
\item
If $M$ is an $R$-module, we say it is $B$-stationary if there is a directed system $(M_{i},f_{ij})$ of finitely presented $R$-modules with $M=\rlim M_{i}$ such that $(M_{i},f_{ij})$ is $B$-stationary. 
\item 
If $\mc{B}$ is a class of $R$-modules, we say $M$ is $\mc{B}$-\ti{stationary} if it is $B$-stationary for all $B\in\mc{B}$.
\end{enumerate}
\end{defn}

\begin{prop}\label{lml}
Let $R$ be a Gorenstein ring, then every module of finite injective dimension is $\lcm{R}$-Mittag-Leffler. Moreover, a module is $\lcm{R}$-Mittag-Leffler if and only if it is $\t{GInj-}R$-stationary.
\end{prop}
\begin{proof}
As $R$ is Gorenstein, there is a cotorsion pair $(\mc{I},\t{GInj-}R)$, where $\mc{I}$ denotes the class of modules of finite injective dimension and $\t{GInj-}R$ is the category of Gorenstein injective modules (see \cite[§10.1]{rha} for more details). Since $R$ is Gorenstein, the class $\t{GInj-}R$ is closed under direct limits \cite[Lemma 11.1.2]{rha}, and therefore direct sums. As $R$ is Gorenstein, the class $\mc{I}$ is equal to the class of all modules of finite projective dimension, so we can conclude from \cite[Prop. 4.1]{bh} that the cotorsion pair $(\mc{I},\t{GInj-}R)$ is of finite type. Moreover, $M\in \lcm{R}$ if and only if $\t{Tor}_{1}(I,M)=0$ for every $I\in\mc{I}$. The result then follows immediately from \cite[Prop. 9.2]{ahh} and \cite[Theorem 9.5]{ahh}.
\end{proof}
\begin{rmk}
The above proof does not use any property of $R$ apart from it having finite injective dimension over itself. Consequently the above result holds for any Iwanaga-Gorenstein ring if one replaces $\cm{R}$ with the class of finitely generated Gorenstein projective $R$-modules.
\end{rmk}

\section{The dimension one case}
\ni
For this section, we will let $(R,\mf{m},k)$ denote a one-dimensional Cohen-Macaulay ring. Since $\t{dim }R=1$, we have $\lcm{R}=\hcm{R}$ even when $R$ does not admit a canonical module. Since $\lcm{R}=\{M\in\Mod{R}:\Hom_{R}(k,M)=0\}$, it is clear that $\lcm{R}$ is closed under submodules. Several interesting phenomena occur in this situation that do not occur in higher dimensional cases, for example, since $\Hom_{R}(k,-)$ preserves inverse limits, we immediately get the following result.

\begin{lem}
For $R$ as above, $\lcm{R}$ is closed under inverse limits.
\end{lem}
\ni
As we previously showed, $\lcm{R}$ is in general not closed under inverse limits when $\t{dim }R\geq 2$. Recall that $\lcm{R}_{\infty}$ consists of all $R$-modules of infinite Ext-depth. In this situation $$\lcm{R}_{\infty} = \{M\in\Mod{R}:\Hom_{R}(k,M)=0=\t{Ext}_{R}^{1}(k,M)\}.$$
Recall from \ref{serre} that $\lcm{R}_{\infty}$ sits inside $\lcm{R}$ in the manner of a Serre subcategory. In dimension one, this inclusion enables us to prove the following result, which is the main result of the section.

\begin{thm}\label{abelian}
Let $\t{dim }R=1$, then $\lcm{R}_{\infty}$ is a Grothendieck abelian category.
\end{thm}

\begin{proof}
By \cite[Prop. 5.92]{rot}, in order to show $\lcm{R}_{\infty}$ is abelian it suffices to show that $\lcm{R}_{\infty}$ is closed under direct sums, contains a zero object and if $f:M\ra N$ is a morphism in $\lcm{R}_{\infty}$, then both $\t{Ker }f$ and $\t{Coker }f$ lie in $\lcm{R}_{\infty}$. Clearly the first two hold. Suppose $f:M\ra N$ is a morphism in $\lcm{R}_{\infty}$. There are two associated short exact sequences $\mc{S}_{1}:0\ra \t{Ker }f\ra M \ra \t{Im }f \ra 0$ and $\mc{S}_{2}: \t{Im }f\ra N\ra \t{Coker }f\ra 0$, from which it is clear that both $\t{Ker }f$ and $\t{Im }f$ are elements of $\lcm{R}$. Applying \ref{serre} to $\mc{S}_{1}$ we see that $\t{Ker }f$ and $\t{Im }f$ are both in $\lcm{R}_{\infty}$. Applying $\Hom_{R}(k,-)$ to $\mc{S}_{2}$ shows that $\t{Coker }f$ also has infinite Ext-depth. This shows that $\lcm{R}_{\infty}$ is closed under kernels and cokernels, so is an Abelian category. Since $\lcm{R}_{\infty}$ is definable, it is closed under coproducts, so is cocomplete, and products, so is complete. Suppose 
$$\{0\ra L_{i}\ra M_{i}\ra N_{i}\ra 0\}_{I}$$
 is a directed system of short exact sequence with terms in $\lcm{R}_{\infty}$, then it is also a directed system in $\Mod{R}$ whose direct limit is the short exact sequence $\mbb{S}:0\ra \rlim L_{i}\ra \rlim M_{i}\ra \rlim N_{i}\ra 0$. Yet all three terms of this exact sequence lie in $\lcm{R}_{\infty}$, so $\mbb{S}$ is actually short exact sequence in $\lcm{R}_{\infty}$. Lastly, we have to show that $\lcm{R}_{\infty}$ contains a generator. Since $\lcm{R}_{\infty}$ is definable, there is a set of objects $\mc{X}$ such that every object in $\lcm{R}_{\infty}$ can be realised as the direct limit of a directed system in $\mc{X}$ (this is a consequence of the Downwards L\"{o}wenheim-Skolem theorem, see \cite[\S 18.1.4]{psl} for more details). The module $G = \oplus_{X\in\mc{X}} X$ acts as a generator for $\lcm{R}_{\infty}$. Indeed, let $M$ be a module in $\lcm{R}_{\infty}$ and $(X_{i},f_{i,j})_{i,j\in I}$ a directed system in $\mc{X}$ with direct limit $M$. By properties of direct limits, there is a pure epimorphism in $\lcm{R}_{\infty}$
$$\bigoplus_{i\in I}X_{i}\ra M.$$
There is then a projection $G^{(I)}\ra \oplus_{I}X_{i}$ and we may compose with $\pi$ to obtain the required surjection $G^{(I)}\ra M$.
\end{proof}
\ni
There is another way $\hcm{R}_{\infty}$ sits inside $\hcm{R}$ which is also specific to the dimension one case. For this, we will need some definitions from exact categories. As both $\hcm{R}$ and $\hcm{R}_{\infty}$ are extension closed, we can view them as exact categories where the exact structure is inherited from $\Mod{R}$. We will say that $L\hra M\tra N$ is a conflation in $\hcm{R}$ if $0\ra L\ra M\ra N\ra 0$ is exact in $\Mod{R}$, and similarly for $\hcm{R}_{\infty}$. We will say that a map $L\ra M$ is an \ti{admissible monomorphism} if it arises in a conflation $L\hra M\tra N$, and we similarly define \ti{admissible epimorphism}. 

\begin{defn}{\cite[Def. 2.15]{tate}}
Let $\mc{D}$ be an exact category. An exact full subcategory $\mc{C}\subset \mc{D}$ is \ti{left filtering} if every morphism $X\ra F$ in $\mc{D}$, with $X\in\mc{C}$, factors through an admissible monomorphism $X\hra F$, with $X'\in\mc{C}$:
$$
\begin{tikzcd}
X \arrow{r} \arrow{dr} & F \\
& X' \arrow[hook]{u}
\end{tikzcd}
$$
\end{defn}

\begin{prop}
When $\t{dim }R=1$, $\hcm{R}_{\infty}$ is left filtering in $\hcm{R}$.
\end{prop}
\begin{proof}
Let $f:X\ra Y$ be a morphism in $\hcm{R}$ with $X\in\hcm{R}_{\infty}$. Then there are two short exact sequences of $R$-modules
$$0 \ra \t{Ker}(f)\ra X \ra \t{Im}(f)\ra 0$$
$$0 \ra \t{Im}(f)\ra Y \ra \t{Coker}(f)\ra 0.$$
Since $\hcm{R}$ is closed under submodules, we see that $\t{Ker}(f)$ and $\t{Im}(f)$ lie in $\hcm{R}$, and therefore $\t{Ker}(f)\hra X\tra \t{Im}(f)$ is a conflation in $\hcm{R}_{\infty}$ by applying \ref{serre}. In particular, $\t{Ext}^{1}(k,\t{Im}(f))=0$ in $\Mod{R}$. Applying the functor $\Hom(k,-)$ to the second exact sequence and then applying the depth lemma (\ref{dl}) shows that $\t{Coker}(f)\in
\hcm{R}$, so $\t{Im}(f)\hra Y \tra \t{Coker}(f)$ is a conflation in $\hcm{R}$. Therefore $f:X\ra Y$ through the admissible monomorphism $\t{Im}(f)\hra Y$.
\end{proof}

\begin{rmk}
The assumption of $\t{dim }R=1$ is necessary for this result: if $\t{dim }R>1$, then in general $\hcm{R}$ will not be closed under submodules, so one cannot usually form the conflation $\t{Ker}(f)\hra X\tra \t{Im}(f)$ in $\hcm{R}$, let alone $\hcm{R}_{\infty}$.
\end{rmk}

\end{document}